\documentclass{amsart}
\usepackage{latexsym}
\usepackage{amsmath}
\usepackage{amsfonts}
\usepackage{amssymb}
\usepackage{amsthm}
\usepackage{longtable}
\usepackage{multirow}
\usepackage{alltt}
\newtheorem{thm}{Theorem}[section]

\newtheorem{remark}{Remark}[section]
\newtheorem{lemma}{Lemma}[section]
\newtheorem{prop}{Proposition}[section]

\newtheorem{exam}{Example}[section]
\newtheorem{defn}{Definition}[section]
\newtheorem{cor}{Corollary}[section]
\newtheorem{conjecture}{Conjecture}[section]

\begin{document}

\title[Complete Positivity]{The DJL Conjecture for CP Matrices Over Special Inclines}
\author[Preeti Mohindru]{Preeti Mohindru}
\author[Rajesh Pereira]{Rajesh Pereira}
\address{Department of Mathematics \& Statistics, University of Guelph, Guelph, ON, Canada N1G 2W1}
\email{preeti.mohindru@gmail.com}
\email{pereirar@uoguelph.ca}
\subjclass[2010]{15A80, 15B48, 16Y60}
\keywords{matrices, inclines, completely positive matrices, totally ordered inclines, diagonally dominant matrices}
\thanks{The authors would like to thank the referees for their suggestions.  P.~Mohindru would like to acknowledge the Government of Ontario for support in the form of an Ontario Government Scholarship. R.~Pereira would like to acknowledge the NSERC for support in the form of Discovery grant 400550.}

\begin{abstract}
Drew, Johnson and Loewy conjectured that for  $n \geq 4$, the CP-rank of every $n \times n$ completely positive real matrix is at most $\left[n^{2}/4\right]$. While this conjecture has recently been disproved for completely positive real matrices, we show that this conjecture is true for $n \times n$ completely positive matrices over certain special types of inclines. In addition, we prove an incline version of Markham's theorems which gives sufficient conditions for completely positive matrices over special inclines to have triangular factorizations.

\end{abstract}
\maketitle

\section{Introduction}

In this paper, we give a characterization of completely positive matrices over special inclines and we find the upper bound on the CP-rank of these matrices.  This upper bound verifies the analog of the Drew-Johnson-Loewy conjecture for matrices over these inclines.  In the next subsection, we review the theory of completely positive real matrices.  In the second subsection,  we look at some of the theory relating the nonnegativity of almost principal minors and triangular decomposition for completely positive real matrices. In the third subsection, we study the theory of semirings.  In our final introductory subsection, we review the theory of inclines.  In section two,  we find necessary and sufficient conditions for symmetric matrices over special inclines to be completely positive.  In section three,  we prove the truth of the Drew-Johnson-Loewy conjecture for completely positive matrices over certain special types of inclines. In section four, we prove analogs of the results relating the nonnegativity of almost principal minors and the triangular decomposition for completely positive matrices over special inclines.

\subsection{Completely  Positive Matrices}

 An $n \times n$ real matrix $A$ is called completely positive (CP) if, for some $m\in \mathbb{N}$, there exists an $n \times m$ nonnegative matrix $B$ such that $A = BB^{T}$. Such a decomposition is called a \textit{completely positive decomposition} and is not necessarily unique. The set of completely positive $n \times  n$ matrices is denoted by $CP_{n}$.  For more details on real completely positive matrices see \cite{BSM03}.

It can be easily seen that $A$ is a real completely positive matrix if and only if $A$ can be written as
\begin{center}
$A =  b_{1}b_{1}^{T} +  b_{2}b_{2}^{T} +...+ b_{m}b_{m}^{T}$
\end{center}
where $b_{i} \in R^{n}$ is nonnegative, for all $i = 1,...,m$. Here $b_{i}$ corresponds to the $i^{th}$  column of $B$ for all $i$ and all $B_{i}^{'}s = b_{i}b_{i}^{T}$  $(i= 1,2,...,m)$ are rank one CP matrices. We refer to this as a \textit{rank 1 CP-representation of $A$}. For $0 \leq b_{j} \in R^{n}$, let supp$(b_{j}) = \{ i : b_{ij} \neq 0\}$ denote the support of $b_{j}$, and let $s(b_{j})$ denote the cardinality of supp$(b_{j})$. A representation
\begin{center}
$A = \sum_{j =1}^{k} b_{j}b_{j}^{T}, \quad \quad b_{j} \geq 0, \quad \quad s(b_{j}) \leq t$
\end{center}
will be called \textit{support t rank 1 CP-representation.}
\begin{defn} Let A be an $n \times n$ real completely positive   matrix. The minimal $m$ such that $A = BB^{T}$ for some nonnegative $n \times m$  real matrix $B$,  is called the \textit{CP-rank} of $A$. The CP-rank of A is denoted by CP-rank$(A)$.
 \end{defn}

Completely positive matrices have applications to such areas as the theory of inequalities,  the theory of block designs in combinatorics,  probability and statistics, optimization theory and economic modeling.

It would be of great interest to have an efficient algorithm to decide if a given matrix is completely positive or an efficient algorithm for computing the CP-rank of a given completely positive matrix.  While there is no efficient way of solving either problem for real matrices, we will see   that for special inclines there is an easy test for complete positivity.

Until recently, the most famous open problem in the theory of completely positive matrices is the following conjecture stated by Drew, Johnson and Loewy.
\begin{conjecture}  \cite{DJL} If $A$ is a real CP-matrix of order $n \geq 4$ then CP-rank($A$) $\leq \left[n^{2}/4\right]$. \end{conjecture}
Here $\left[x\right]$ is the greatest integer function.

 This conjecture  has been listed as a problem by Xingzhi Zhan in  "Open Problems in Matrix Theory" \cite{Zh08}.  It has been proven for certain special cases. If $A$ is an $n \times n$ symmetric matrix, the graph of $A$, denoted by $G(A)$, is a graph on vertices $1, 2, . . . , n$ with
$\{i, j\}$ an edge if and only if $i \neq j$ and $a_{ij} \neq 0$.
\begin{defn}   Let $G$ be a graph on $n$ vertices  and  $A$ be an $n \times n$ real completely positive matrix. The matrix $A$ is called a \textit{CP matrix realization} of $G$ if  $G(A) = G$.  \end{defn}
\begin{defn} \cite{NSM} Let $G$ be a graph on $n$ vertices. The CP-rank of $G$, denoted by
CP-rank$(G)$, is the maximal CP-rank of a CP matrix realization of $G$, that is,\begin{center}
CP-rank$(G)$ = max$\{$CP-rank$(A) | A$ is CP and $G(A) = G\}$.\end{center}
\end{defn}
The Drew-Johnson-Loewy conjecture can be rephrased  \cite{NSM} as: for every graph $G$ on $n \geq 4$ vertices, CP-rank($G$) $\leq \left[n^{2}/4\right]$. This conjecture has been proven for triangle free graphs in \cite{DJL}, for graphs which contain no odd cycle of length 5 or more in \cite{DrJ}, for all graphs on 5 vertices which are not the complete graph in \cite{LoT}, for nonnegative matrices with a positive semidefinite comparison matrix (and any graph) in \cite{BSM}  and for all $5 \times 5$ completely positive matrices  in \cite{SBJS}.   However,  Bomze, Schachinger and Ullrich \cite{BSU14, BSU15} have disproved the  Drew-Johnson-Loewy conjecture for real completely positive matrices by constructing counterexamples in every dimension greater than or equal to seven.

\subsection{Almost Principal Minors}

A minor of a matrix $A$ is the determinant of a square submatrix of $A$.  We remind the reader of the standard notation for submatrices.
 Let $\alpha=\{\alpha_1<\alpha_2<...<\alpha_k\}$ and $\beta=\{\beta_1<\beta_2<...<\beta_k\}$ be two subsets of $\{ 1,2,...,n\} $ of cardinality $k$.  Then $A[\alpha |\beta]$ is the $k$ by $k$ submatrix of $A$ whose $(i,j)$th entry is $a_{\alpha_i\beta_j}$ and $det(A[\alpha |\beta])$ is a minor of $A$.  The set of all minors of $A$ is $\{ det(A[\alpha |\beta]): \alpha ,\beta \subseteq \{ 1,2,...,n\}, |\alpha| =|\beta|\}$. A minor $det(A[\alpha |\beta])$ is called a principal minor if $\alpha=\beta$.  Positive semidefinite matrices, a class which includes all completely positive matrices, have nonnegative principal minors. There are two other classes of minors which play a key role in the theory of complete positivity.

 \begin{defn} Let $A \in M_{n}(\mathbb{R})$. Let $\alpha=\{\alpha_1<\alpha_2<...<\alpha_k\}$ and $\beta=\{\beta_1<\beta_2<...<\beta_k\}$ be two subsets of $\{ 1,2,...,n\} $ of the same cardinality. Then the $k \times k$ submatrix, $A[\alpha| \beta]$ is called a \textit{left almost principal submatrix} of $A$ if $\alpha_j=\beta_j$ for all $j:2\le j\le k$ but $\alpha_1\neq \beta_1$.  The determinant of a left almost principal submatrix is called a \textit{left almost principal minor}.
\end{defn}

\begin{defn}  Let $A \in M_{n}(\mathbb{R})$. Let $\alpha=\{\alpha_1<\alpha_2<...<\alpha_k\}$ and $\beta=\{\beta_1<\beta_2<...<\beta_k\}$ be two subsets of $\{ 1,2,...,n\} $ of the same cardinality. Then the $k \times k$ submatrix, $A[\alpha| \beta]$ is called a \textit{right almost principal submatrix} of $A$ if $\alpha_j=\beta_j$ for all $j:1\le j\le k-1$ but $\alpha_k\neq \beta_k$.  The determinant of a right almost principal submatrix is called a \textit{right almost principal minor}.\end{defn}

The almost principal minors play a key role in the theory of nonnegative LU and UL decompositions.

\begin{defn} We say that an $n$ by $n$ real matrix $A$ is $UL$-completely positive if there exists an $n$ by $n$ upper triangular nonnegative matrix $B$ such that $A=BB^T$.  We say that $A$ is $LU$-completely positive if there exists an $n$ by $n$ lower triangular nonnegative matrix $C$ such that $A=CC^T$. \end{defn}

The following results are due to Markham \cite{TLM} and can also be found in the reference \cite{BSM03}.

\begin{thm} \cite{TLM} \label{Mar1} Let $A$ be a doubly nonnegative matrix.  If all of the left almost principal minors of $A$ are nonnegative then $A$ is $UL$-completely positive.
\end{thm}

\begin{thm} \cite{TLM} \label{Mar2} Let $A$ be a doubly nonnegative matrix.  If all of the right almost principal minors of $A$ are nonnegative then $A$ is $LU$-completely positive.
\end{thm}

It has been shown in \cite{BSM03}, that for $n \leq 3$, every $n \times n$  completely positive real matrix is either LU-completely positive or UL-completely positive or both LU-completely positive and UL-completely positive. However for $n \geq 4$, an $n \times n$ completely positive real matrix may be neither  UL-completely positive nor  LU-completely positive.  An example is the matrix below given in \cite[Example 2.17]{BSM03}.
For example, the matrix \begin{center}
$A = \begin{bmatrix}
2&0&0&1&1\\0&2&0&1&1\\0&0&2&1&1\\1&1&1&2&0\\1&1&1&0&2
\end{bmatrix}$
\end{center}
is a completely positive real matrix, but it is neither UL-completely positive nor  LU-completely positive.

Note that if a real matrix $A$ is LU-completely positive, then it does not necessarily follow that $A$ is UL-completely positive. For example, consider the matrix from \cite[Example 2.19]{BSM03},
\begin{center}
$A = \begin{bmatrix}
2&1&3\\1&3&2\\3&2&5
\end{bmatrix}$
\end{center}
A is LU-completely positive  and the LU-completely positive factorization of $A$ is
\begin{center}
$A= \begin{bmatrix}
\sqrt{2}&0&0\\\frac{1}{\sqrt{2}}&\frac{\sqrt{5}}{\sqrt{2}}&0\\\frac{3}{\sqrt{2}}&\frac{1}{\sqrt{10}}&\frac{2}{\sqrt{10}}
\end{bmatrix}\begin{bmatrix}
\sqrt{2}&0&0\\\frac{1}{\sqrt{2}}&\frac{\sqrt{5}}{\sqrt{2}}&0\\\frac{3}{\sqrt{2}}&\frac{1}{\sqrt{10}}&\frac{2}{\sqrt{10}}
\end{bmatrix} ^{T}$
\end{center}
 but it can be verified that $A$ is not UL-completely positive.

The following example shows that a UL-completely positive real matrix may not be LU-completely positive. Let
\begin{center}
$A = \begin{bmatrix}
2&1&3\\1&3&1\\3&1&5
\end{bmatrix}$
\end{center}
It is UL-completely positive and its UL-completely positive factorization of $A$ is
\begin{center}
$A= \begin{bmatrix}
\sqrt{\frac{5}{35}}&\sqrt{\frac{2}{35}}&\frac{3}{\sqrt{5}}\\0&\sqrt{\frac{14}{5}}&\frac{1}{\sqrt{5}}\\0&0&\sqrt{5}
\end{bmatrix}\begin{bmatrix}
\sqrt{\frac{5}{35}}&\sqrt{\frac{2}{35}}&\frac{3}{\sqrt{5}}\\0&\sqrt{\frac{14}{5}}&\frac{1}{\sqrt{5}}\\0&0&\sqrt{5}
\end{bmatrix} ^{T}$
\end{center} but it can be verified that $A$ is not LU-completely positive.

A real matrix is called \textit{totally nonnegative} if all of its minors are nonnegative.  The class of nonnegative matrices is an area of great interest.  We note that the following is an immediate consequence of Markham's theorems.

\begin{cor} \label{marcor} Any square symmetric totally nonnegative matrix is both LU and UL-completely positive.\end{cor}

\subsection{Semirings}

Semirings are a natural generalization of rings.  Semirings satisfy all properties of unital rings except the existence of additive inverses. H. S. Vandiver  introduced the concept of semiring in \cite{VAN34}, in connection with the axiomatization of the arithmetic of the natural numbers.
If in a semiring $S$ the multiplication operation ($\otimes$) is   commutative  then $S$ is called a \textit{commutative semiring}.  A semiring is said to be \textit{antinegative} or \textit{zerosumfree} if the only element with an additive inverse is the additive identity $\textbf{0}$.   An element $a \in S$ is said to be additively (resp. multiplicatively) \textit{idempotent} if $a \oplus a = a$ (resp. $a \otimes a = a$). A semiring $S$ is said to be additively (resp. multiplicatively) \textit{idempotent} if every element of $S$ is additively (resp. multiplicatively) idempotent.

The \textit{nonnegative real numbers} under the usual addition and multiplication form a semiring.   A much studied example of a semiring is  the \textit{max-plus semiring} \cite{MPLUS90}, where $\mathbb{R}_{max} = \mathbb{R}\bigcup\{-\infty\}$ with   $a\oplus b = \max\{a,b\}$ and  $a\otimes b = a+ b$.  Note that in this case  \textbf{0} = -$\infty$ and \textbf{1} = 0.    A survey of some combinatorial applications of the max-plus semiring can be found in \cite{Bu03}.  A totally ordered set S with greatest element \textbf{1} and least element \textbf{0} forms a semiring  \cite{GOL90}, with  $a \oplus b = max\{a,b\}$ and $a \otimes b = min\{a,b\}$. This is called a \textit{max-min semiring}. Max-min semirings are sometimes called \textit{chain semirings}.

A  \textit{Boolean algebra} $\textit{B}$  with a  unique minimal element $\textbf{0}$, a unique maximal element $\textbf{1}$, forms a semiring   where addition and multiplication is defined as   $ a \oplus b = a \cup b $ and $ a \otimes b = a \cap b$.  Here $\cap$  denotes the   \textit{intersection} operation  and  $\cup$ denotes the \textit{union}  operation. Any  distributive lattice with a  unique minimal element $\textbf{0}$ and  a unique maximal element $\textbf{1}$  forms  a semiring under addition and multiplication defined as
  $a \oplus b = a \vee b = l.u.b\{a,b\}$ and $a \otimes b = a \wedge b = g.l.b\{a,b\}$.

All above examples of semirings are both commutative and antinegative.




There is one other property which is useful in semirings.

\begin{defn} $(\text{The Unique Square Root Property})$  A semiring $S$ is said to have \textit{the unique square root property} if for  any $x \in S$ there exists a unique $c \in S$ such that $x = c^{\otimes 2}$, where $c^{\otimes 2} = c \otimes c$. We also write this as  $c=\sqrt{x} $.\end{defn}

The nonnegative real numbers, max-min semirings, the max-plus semiring and distributive lattices all have the unique square root property while the natural numbers and the real numbers are examples of semirings without the unique square root property.

The concepts of matrix theory are defined for matrices over a semiring in a similar way to which they are defined for matrices over a field. If A = $(a_{ij})$ is an n by n matrix over a commutative  ring, then
the standard determinant expression of A is  \cite{PH04}:
\begin{center}
$det(A) = \sum_{\sigma \in S_{n}} sgn(\sigma)a_{1 \sigma(1)}a_{2 \sigma(2)}$......$a_{n \sigma(n)}$
\end{center}

where $S_{n}$ is the symmetric group of order n and $sgn(\sigma) = +1$ if  $\sigma$ is even permutation and $sgn(\sigma) = -1$ if  $\sigma$ is odd permutation.  Here $sgn(\sigma)a_{1 \sigma(1)}a_{2 \sigma(2)}$......$a_{n \sigma(n)}$ is called a term of the determinant.

Since we do not have subtraction in a semiring, we can not write the determinant of a matrix over a semiring in this form. We split the determinant into two parts, the positive determinant and the negative determinant.
\begin{defn}   
Let $A$ be an $n$ by $n$ matrix over a commutative semiring $S$,  then we define the positive and the negative
determinant as:
\begin{center}
$det^{+}(A) = \bigoplus_{\sigma \in A_{n}} \bigotimes_{i = 1}^{n}a_{i \sigma(i)}$
\end{center}
\begin{center}
$det^{-}(A) = \bigoplus_{\sigma \in S_{n}\backslash A_{n}} \bigotimes_{i =1}^{n} a_{i \sigma(i)}$
\end{center}
Where $A_{n}$ is the alternating group of order $n$,  i.e,  the set of
all even permutations of order $n$ and $S_{n}\backslash A_{n}$ is the set  of
all odd permutations of order $n$.
\end{defn}
As such we note that the determinant of a matrix $A$ over a ring takes
the form:
\begin{center}
$det(A) = det^{+}(A) - det^{-}(A)$
\end{center}
Many of the properties of positive and negative determinants over semirings can be found in \cite{PH04}.

In \cite{PMTA}, the     Drew-Johnson-Loewy conjecture was generalized to completely positive matrices over semirings and was proved for completely positive matrices over max-min semirings. Although the original Drew-Johnson-Loewy conjecture was disproved, the generalized  Drew-Johnson-Loewy conjecture   is still open for many other  semirings. In this paper, we prove  the truth of the  Drew-Johnson-Loewy conjecture for completely positive matrices over certain special types of  semirings.



\subsection{Inclines}

The incline is an algebraic structure which was first introduced under the name of slope by Cao \cite{Cao83}. It was given its modern name of incline by Cao, Kim and Roush \cite{CKR84} which remains the authoritative reference on the theory of inclines and their applications. More recently, Kim and Roush \cite{KR04} have surveyed and described algebraic properties of inclines and matrices over inclines.

\begin{defn}   \cite{CKR84}  $(\text{Inclines})$ A nonempty set $L$ with two binary operations $\oplus$ and $\otimes$ is called an incline if it satisfies the following conditions;
\begin{enumerate}
\item $(L, \oplus = l.u.b)$ is a semilattice.
\item $(L, \otimes)$ is a  semigroup.
\item $x \otimes (y \oplus z) = (x \otimes y) \oplus (x \otimes z)$,  \qquad for all $x, y, z \in L$.
\item $x \oplus (x \otimes y) = x$, \qquad for all $x, y \in L$.
\end{enumerate} \end{defn}

An incline $L$ is called a  \textit{commutative incline} if $(L, \otimes)$ is a  commutative semigroup. If an incline $L$ has the additive identity $\textbf{0}$, then it follows that $\textbf{0}$ is the least element of $L$ and $\textbf{0} \otimes x = x\otimes \textbf{0}= \textbf{0}$ for all $x \in L$.  Similarly if $L$ has a multiplicative identity $\textbf{1}$, it follows that $\textbf{1}$ is the greatest element of $L$ and $\textbf{1} \oplus x = \textbf{1}$ for all $x \in L$.  If $L$ lacks an additive or multiplicative identity, these may be added to $L$.

In an incline $L$, define a relation $\leq$ by
\begin{center}
$x \leq y \qquad \Leftrightarrow \qquad x \oplus y = y$.
\end{center}
This is a partial order relation. An incline is said to be \textit{linearly ordered} or \textit{totally ordered} if the partial order relation $\leq$ is a total order relation. Vector spaces over totally ordered inclines have been studied in \cite{DGZXT11}.  We note that product of any two elements is less than or equal to either of the elements.  That is $x\otimes y\le x$ and $x\otimes y\le y$ for all $x,y\in L$.

 Examples of totally ordered inclines include the two element Boolean semiring ($\{\textbf{0}, \textbf{1}\}$), the max-min semiring ($[0,1], max(x,y), min(x,y) )$, the negative interval subsemiring of the max-plus semiring ($[-\infty, 0], max(x,y), x+y$),  and the max-times   semirings, ($[0, 1], max(x,y), xy)$ where $xy$ is the ordinary real multiplication.

Distributive lattices and Boolean algebras are  also  inclines which may not be totally ordered.

There are two notions of ideal in incline theory \cite{AK02, KR04}: ideals in the semiring sense and ideals in the lattice sense.

\begin{defn}  $(\text{r-Ideal or Ideal  in Semiring Sense})$ An \textit{r-ideal} $J$ of an incline $L$ is a nonempty subset of $L$ satisfying  the following conditions:
\begin{enumerate}
\item $a \in J$ and $x \in L$ implies that   $x \otimes a \in J$ and $a \otimes x \in J$.
\item $a \in J$ and $b \in J$ implies that $a \oplus b   \in J$.
\end{enumerate}
\end{defn}

\begin{defn}   $(\text{Lattice Ideal or Ideal in Lattice Sense})$ A  \textit{lattice ideal} $J$ of an incline $L$ is a nonempty subset of $L$ satisfying  the following conditions:
\begin{enumerate}
\item $a \in J$ implies that   $x \in J$  for all $x \leq a$, where $x \in L$.
\item $a \in J$ and $b \in J$ implies that $a \oplus b   \in J$.
\end{enumerate}
\end{defn}

It is easy to check that every lattice ideal of a commutative incline is an r-ideal. However, an r-ideal of a commutative incline may not be a lattice ideal. For example, suppose that $L = \{\mathbb{N} \cup \{\infty\}, min, \times\}$, where $\mathbb{N}$ is the set of all natural numbers not including zero. Evidently, $L$ forms a  commutative  incline where  $\infty$ is the additive identity and $1$ is the multiplicative identity. Note that in  this incline $L$, the order relation is reversed. The set of all even natural numbers not including zero forms an r-ideal but it is not a lattice ideal.  We note that this example shows that even singly generated r-ideals may not be lattice ideals.

In this paper, we will be particularly interested in inclines whose singly generated r-ideals are all lattice ideals.
\begin{defn}
 $(\text{The LI-Property})$ A commutative incline $L$ is said to have \textit{the LI-property} if all singly generated r-ideals of $L$ are lattice ideals of $L$. \end{defn}

We know that every lattice ideal of a commutative incline $L$ is an r-ideal of $L$. If in a commutative incline $L$, all singly generated  r-ideals  are lattice ideals  then we get that all singly generated  r-ideals of $L$ are same as the singly-generated lattice ideals of $L$. In the following proposition we characterize those commutative inclines which have the  LI-property.

\begin{prop} A  commutative incline $L$ has the LI-property, i.e.,  all singly generated r-ideals of $L$ are lattice ideals   of $L$ if and only if  for $x, y \in L$,    $x \leq y$ implies that there exists  $z \in L$  (not necessarily unique)  such that $x = y\otimes z$.  \end{prop}

\begin{proof}  Let $y$ be an arbitrary element of $L$ and $I$ be the $r$-ideal of $L$ generated by $y$.  It is clear that $y$ is the largest element of $I$.  Then $I$ is a lattice ideal if and only if $x\le y$ implies that $x\in I$.  Since $x\in I$ if and only if there exists  $z \in L$  such that $x = y\otimes z$, our result follows.
\end{proof}

We now discuss some properties of inclines with the unique square root property.
Note that the uniqueness  of  the square root for these commutative inclines  implies that the square root function is multiplicative, i.e.,
\begin{center} $\sqrt{x}\otimes  \sqrt{y} = \sqrt{x\otimes y}$. \end{center}
Let $x, y, c, d \in L$, such that $x = c^{\otimes 2}$ and $y = d^{\otimes 2}$. This implies that
\begin{equation}\label{eq_sq_multi_1} x\otimes y = c^{\otimes 2}d^{\otimes 2} = (c\otimes d)^{\otimes 2}, \text{ i.e. },  \sqrt{x\otimes y} = c\otimes d. \end{equation} We also have $\sqrt{x} = c$ and $\sqrt{y} = d$. This implies that
\begin{equation}\label{eq_sq_multi_2}
\sqrt{x}\otimes \sqrt{y} = c\otimes d.
\end{equation}
From \eqref{eq_sq_multi_1} and \eqref{eq_sq_multi_2} we get that
\begin{equation}\label{eq_sq_multi_3}
\sqrt{x}\otimes  \sqrt{y} = \sqrt{x\otimes y}.
\end{equation}

Furthermore, we note that if a  commutative  incline $L$ has the unique square root property and the LI-property then   the square root function is an increasing function, i.e.,
\begin{center}
$\text{if } x \leq y \quad \Longrightarrow \quad \sqrt{x} \leq \sqrt{y}$.
\end{center}
Let $x, y  \in L$, such that  $x \leq y$. This implies that
\begin{equation}\label{eq_cancellative_property}
x = y\otimes z, \qquad \text{where }  z \in L, \qquad \text{(using the LI-property)}.
\end{equation} Taking square root on both sides of \eqref{eq_cancellative_property}, we get
\begin{center}
$\sqrt{x} = \sqrt{y\otimes z}$ \qquad \qquad \qquad \qquad \qquad \quad \qquad \qquad \qquad
\end{center} \begin{center}
 $\Rightarrow \sqrt{x} = \sqrt{y}\otimes \sqrt{z}, \qquad \text{ (using \eqref{eq_sq_multi_3})}$ \qquad \qquad \qquad \qquad \qquad
\end{center} \begin{center}
 $\Rightarrow \sqrt{x} \leq \sqrt{y} \qquad \text{ (using the LI-property)}.$
\end{center}

One final useful property is a version of the arithmetic-geometric mean inequality for elements of the incline.

\begin{defn} $(\text{Arithmetic Geometric Property})$ A commutative incline $L$ is said to have the arithmetic geometric property $($AG-property$)$ if $x \otimes y \leq x^{\otimes 2} \oplus y^{\otimes 2}$.
\end{defn}
Note that every totally ordered commutative incline has the arithmetic geometric property. Since in every totally ordered commutative incline $L$ either $x \leq y$ or $y \leq x$, for all $x, y \in L$. This implies that either $x \otimes y \leq y^{\otimes 2}$ or $x \otimes y \leq x^{\otimes 2}$. Therefore, $x \otimes y \leq l.u.b\{x^{\otimes 2}, y^{\otimes 2}\} = x^{\otimes 2} \oplus y^{\otimes 2}$. Moreover, commutative inclines in which the multiplication is idempotent also have the arithmetic geometric property, since $x \otimes y \leq x = x^{\otimes 2}$ and  $x \otimes y \leq y = y^{\otimes 2}$. Therefore, $x \otimes y \leq l.u.b\{x^{\otimes 2}, y^{\otimes 2}\} = x^{\otimes 2} \oplus y^{\otimes 2}$.

\begin{prop} \label{AG_property_and_sum_of_squares_and_squares_of_sums} Let $L$ be a commutative incline with the arithmetic geometric property. Then $\bigoplus\limits_{i = 1}^{k} x_{i}^{\otimes 2} = \bigg(\bigoplus\limits_{i= 1}^{k} x_{i} \bigg)^{\otimes 2}$, where $x_{i} \in L$ for all $i$.
\end{prop}
\begin{proof} It is evident that the result is true for $k = 1$. For $k = 2$, we have to prove that $x_{1}^{\otimes 2} \oplus x_{2}^{\otimes 2} = (x_{1} \oplus x_{2})^{\otimes 2}$. Since $L$  has the arithmetic geometric property,  $l.u.b\{x_{1}^{\otimes 2}, x_{2}^{\otimes 2}, (x_{1} \otimes x_{2})\} = l.u.b\{x_{1}^{\otimes 2}, x_{2}^{\otimes 2}\}$. Thus we have $(x_{1} \oplus x_{2})^{\otimes 2} = x_{1}^{\otimes 2} \oplus x_{2}^{\otimes 2} \oplus (x_{1} \otimes x_{2}) \oplus (x_{1} \otimes x_{2}) = l.u.b\{x_{1}^{\otimes 2}, x_{2}^{\otimes 2}, (x_{1} \otimes x_{2})\} = l.u.b\{x_{1}^{\otimes 2}, x_{2}^{\otimes 2}\} = x_{1}^{\otimes 2} \oplus x_{2}^{\otimes 2}$. Hence the result is true for $k = 2$. A simple induction argument shows that the result is   true for all $k$.
\end{proof}

\begin{defn} $(\text{Normal Incline})$ A commutative incline $L$ is called a \textit{normal incline} if it has both an additive identity $\textbf{0}$ and a multiplicative identity $\textbf{1}$ and also has the LI-property, the unique square root property and the AG-property.
\end{defn}

Inclines which are normal and totally ordered are called \textit{totally ordered normal} inclines. All above examples of totally ordered  inclines are totally ordered  normal inclines.


We are now ready to define positive semidefiniteness and complete positivity for matrices over inclines.

\begin{defn}  Let $L$ be a commutative incline.  An $n \times n$ matrix $A$ over $L$ is called \textit{positive semidefinite} if $A$ can be written as $A = BB^{T}$, where $B$ is an $n$ by $k$ matrix over $L$.  If further $B$ can be chosen so that every entry of $B$ is a perfect square, then $A$ is called \textit{completely positive}.
\end{defn}

This definition agrees with \cite{PMRP} where these concepts were defined for matrices over general semirings.  If $L$ has the unique square root property, then positive semidefiniteness and complete positivity coincide.  Since most of our results involve normal inclines, we can use either notion.  As most of our results are generalizations of results on real completely positive matrices, we will use the term completely positive for these results to emphasize this connection.

\begin{defn}  Let $L$ be a commutative incline with the unique square root property. The CP-rank of an  $n \times n$ completely positive matrix $A$ over $L$ is the smallest number $k$ such that $A = BB^{T}$, where    $B$ is an   $n \times k$ matrix   over $L$.
\end{defn}
It will also be useful to define the concept of diagonal dominance for matrices over inclines.

\begin{defn}  Let $L$ be a commutative incline.  A matrix  $A \in M_{n}(L)$  is called   diagonally dominant if  $a_{ii} \geq \mathop{\bigoplus\limits_{j = 1}^{n}}\limits_{ j \neq i}  a_{ij}$, for all $i$ such that $1\le i\le n$.
\end{defn}

\section{Characterization of CP Matrices over Special Inclines}

Hannah and Laffey \cite{HaL} remarked that no general necessary and sufficient conditions for a real matrix $A$ to be completely positive are known. Some special results in this respect were obtained by Markham \cite{TLM} and Lau and Markham \cite{LMa}. In particular,  M. Kaykobad \cite{MKK}  has shown that   diagonal  dominance  is a sufficient condition for  real nonnegative  symmetric matrices to be completely positive. It has been shown in \cite{PMTA} that   matrices over  max-min semirings are completely positive if and only if they are  both symmetric and diagonally dominant. We now provide a similar characterization for completely positive matrices over normal inclines.


\begin{thm}\label{thmequivlantincline}  Let $L$ be a normal  incline  and $A \in M_{n}(L)$  be a symmetric matrix. Then the following are equivalent.
\begin{enumerate}
\item $A$ is positive semidefinite.
\item $A$ is completely positive.
\item Every $2$ by $2$ principal submatrix of $A$ has its positive determinant greater than or equal to its negative determinant. $($i.e., $a_{ii}\otimes a_{jj} \geq a_{ij}\otimes a_{ji}$, for all $i,j)$.
\item There exists a diagonal matrix $D \in M_{n}(L)$ and a symmetric   matrix $M \in M_{n}(L)$ all of whose diagonal entries are equal to $\textbf{1}$,   such that $A=DMD$.
\end{enumerate}
\end{thm}
\begin{proof}
 $(1)\implies (2)$ Since $L$ is a normal incline, every element of $L$ is the square of its unique square root.  Therefore any positive semidefinite matrix over $L$ is completely positive.

 $(2)\implies (3)$ Let $A$ be a completely positive matrix  over a normal incline $L$. This implies that  there exists a matrix $B$ over the incline $L$ such that $A = BB^{T}$. Now
\begin{center}
$a_{ii} = \bigoplus\limits_{k}^{} (b_{ik} \otimes b_{ik}) = \mathop{l.u.b}\limits_{k}^{} \left\{ b_{ik}^{\otimes 2}\right\}$
\end{center}
\begin{center}
$a_{jj} = \bigoplus\limits_{k}^{} (b_{jk} \otimes b_{jk}) =  \mathop{l.u.b}\limits_{k}^{} \left\{ b_{jk}^{\otimes 2}\right\}$
\end{center}
\begin{center}
$a_{ij} = \bigoplus\limits_{k}^{} (b_{ik} \otimes b_{jk}) =  \mathop{l.u.b}\limits_{k}^{}  \left\{b_{ik}  \otimes b_{jk}\right\}$
\end{center}
Clearly \begin{align*}  a_{ii} \otimes a_{jj} &= \mathop{l.u.b}\limits_{k}^{} \left\{  b_{ik}^{\otimes 2}\right\} \otimes  l.u.b_{k} \left\{  b_{jk}^{\otimes 2}\right\} \\
&\geq  \mathop{l.u.b}\limits_{k}^{} \left\{ b_{ik}^{\otimes 2} \otimes b_{jk}^{\otimes 2}\right\}\\
 &= \mathop{l.u.b}\limits_{k}^{} \left\{ (b_{ik} \otimes b_{jk})^{\otimes 2}\right\}\\
&= (\mathop{l.u.b}\limits_{k}^{} \left\{ b_{ik} \otimes b_{jk}\right\})^{\otimes 2}, \qquad  (\text{by  proposition } \ref{AG_property_and_sum_of_squares_and_squares_of_sums})\\
&=   a_{ij}^{\otimes 2},\end{align*}
   and this is true for all $i, j$. Hence every $2 \times 2$ principal submatrix of $A$ has $det^{+} \geq det^{-}$.

$(3)\implies (4)$  Let $D=diag(\sqrt{a_{11}}, \sqrt{a_{22}}, \sqrt{a_{33}},...,\sqrt{a_{nn}})$.  Since $a_{ii}\otimes a_{jj} \ge a_{ij}^{\otimes 2}$, we have $\sqrt{a_{ii}}\otimes \sqrt{a_{jj}}\ge a_{ij}$.  Hence by the LI-property,   for any $i\neq j$, there exists $m_{ij} \in L$ such that $\sqrt{a_{ii}}\otimes m_{ij}\otimes \sqrt{a_{jj}}= a_{ij}$.  Let $M$ be the matrix whose main diagonal entries are $\bf{1}$ and whose off-diagonal entries are $m_{ij}$.  A simple calculation shows us that $A=DMD$.

$(4)\implies (1)$ Let us suppose that $M \in M_{n}(L)$ is a symmetric   matrix with all diagonal entries equal to $\textbf{1}$. For $1 \leq k < l \leq n$, construct $n \times n$ matrices $M_{kl}$ over $L$ such that  $(k,k), (k,l), (l,k)$ and $(l,l)$ entry of $M_{kl}$ is $\textbf{1}, m_{kl}, m_{lk}$ and $\textbf{1}$ respectively and all other entries are $\textbf{0}$.  Clearly
\begin{center}$M = \bigoplus\limits_{1 \leq k < l \leq n}^{} M_{kl}$.
\end{center}Further, we can write  $M_{kl} = BB^{T}$, where $B$ is an $n \times 3$ matrix over $L$ whose
\begin{center}  $(k,1)$ entry is = $\sqrt{m_{kl}}$,\end{center}  \begin{center}  $(l,1)$ entry is =  $\sqrt{m_{kl}}$,\end{center}
\begin{center} $(k,2)$ entry is = $\textbf{1}$, \end{center} \begin{center} $(l,3)$ entry is = $\textbf{1}$, \end{center} and all other entries are $\textbf{0}$.
Thus all the matrices  $M_{kl}$, $1 \leq k < l \leq n$, are completely positive. Hence $M$ and $DMD = A$ is completely positive.
\end{proof}

It is clear from theorem \ref{thmequivlantincline} that diagonal dominance is a sufficient condition for  symmetric matrices  over normal  inclines    to be completely positive. It is not a necessary condition for symmetric matrices  over normal  inclines   to be completely positive, since we have matrices over normal  inclines   which are completely positive but not diagonally dominant. One can easily check that the    semiring $\big([-\infty, 0], max, +\big)$   forms a normal  incline.  Here is an example of such a matrix over the semiring  $[-\infty, 0]$:
 \begin{center}
$A = \begin{bmatrix}
-4&-5\\-5&-6\\
\end{bmatrix} = \begin{bmatrix}
-2\\-3
\end{bmatrix}\begin{bmatrix}
-2&-3
\end{bmatrix}$
\end{center}

Clearly $A$ is not a diagonally dominant matrix over $[-\infty, 0]$, but it is completely positive with the  CP-rank equal to one.

Now we will examine a special class of normal inclines in which the  diagonally  dominance condition  is necessary and sufficient  for symmetric matrices to be completely positive.
\begin{defn} \label{defn_regular_incline} $($Regular Incline$)$  \cite{HR13}   An incline $L$ is said to be regular if every element of $L$ is multiplicatively idempotent, i.e., for every $a \in L$, $a \otimes a = a$.
\end{defn}

Examples of regular inclines include the Boolean semiring,  max-min semirings and distributive lattices. It has been shown   \cite[Corollary~3.3]{HR13}   that every regular incline is commutative. Furthermore, every element in a regular incline is the unique square root of itself. Thus  regular inclines   satisfy  the unique square root property.

We also note that  every regular incline has the AG-property, since $x \otimes y \leq x = x^{\otimes 2}$ and  $x \otimes y \leq y = y^{\otimes 2}$. Therefore, $x \otimes y \leq l.u.b\{x^{\otimes 2}, y^{\otimes 2}\} = x^{\otimes 2} \oplus y^{\otimes 2}$. This implies that the  theorem \ref{thmequivlantincline}  holds for all symmetric matrices over regular inclines  having the LI-property.

In the next theorem we will prove that the diagonal dominance condition is necessary and sufficient for   symmetric matrices over  regular inclines  having the LI-property  to be completely positive. Since the Boolean semiring and max-min semirings are regular inclines with the LI-property, the following result is a   generalization the corresponding result for max-min semirings in \cite{PMTA}.

Our  result is formulated as follows:
\begin{thm}  \label{thm_regular_L_diag_dominence} Let $L$ be a regular incline with the LI-property   and $A$ be an $n \times n$  symmetric matrix over $L$. Then the matrix $A$ is positive semidefinte (or equivalently completely positive) if and only if $A$ is diagonally dominant. \end{thm}

\begin{proof} Let  $A$ be a symmetric diagonally dominant matrix over a regular incline  $L$ with the LI-property. This implies that  every $2$ by $2$ principal submatrix has its positive determinant greater than or equal to the negative determinant. Therefore, $A$ is a positive semidefinite matrix over $L$, by  theorem \ref{thmequivlantincline}.

For the other direction, suppose $A$ is a positive semidefinite matrix over a regular incline  $L$. This implies that there exists a matrix $B$ over the   incline $L$ such that $A = BB^{T}$. Thus we get,
 \begin{align*}
  a_{ii} &= \bigoplus\limits_{k=1}^{n} (b_{ik} \otimes b_{ik}) \\
&= \bigoplus\limits_{k=1}^{n} b_{ik}  \\
&\geq \bigoplus\limits_{k=1}^{n} (b_{ik} \otimes  b_{jk}) \\
&= a_{ij}
\end{align*}
and this is true for all $i, j$. This implies that $a_{ii} \geq \mathop{l.u.b}\limits_{j  \neq i}^{} \{a_{ij}\} = \mathop{\oplus}\limits_{j \neq i}^{} a_{ij}$. Hence $A$ is diagonally dominant.
\end{proof}

\section{The CP-rank of CP matrices over Special Inclines}

In this section, we prove  the Drew-Johnson-Loewy conjecture for completely positive matrices over  totally ordered normal  inclines.  We start with the following lemma.

\begin{lemma} \label{inclines_cp_2_3}  Let $L$ be a totally ordered normal  incline  and $A$ be an $n \times n$ completely positive matrix over $L$. Then the CP-rank$(A) \leq n$ for $n = 2,3$.
\end{lemma}
\begin{proof} Let  $A$ be an $n \times n$    completely positive matrix over a totally ordered normal  incline $L$. Because of theorem \ref{thmequivlantincline}, we only need to consider symmetric matrices over $L$ all of whose diagonal entries are $\bf{1}$. For $n = 2$, let us consider
 \begin{center}
$A = \begin{bmatrix}
\bf{1}& a_{12}\\
a_{12}& \bf{1}\\
\end{bmatrix}$
\end{center}
Where $a_{12} \in L$, then the rank 1 CP-representation of $A$ is
\begin{center}
$A = \begin{bmatrix}
\bf{1}\\a_{12}
\end{bmatrix}\begin{bmatrix}
\bf{1} & a_{12}
\end{bmatrix} \bigoplus \begin{bmatrix}
\bf{0}\\ \bf{1}
\end{bmatrix}\begin{bmatrix}
\bf{0}& \bf{1}\\
\end{bmatrix}$
\end{center}
\begin{center}
= $\begin{bmatrix}
 \bf{1} &a_{12}\\a_{12}& a_{12}^{\otimes 2}
\end{bmatrix} \bigoplus \begin{bmatrix}
\bf{0}& \bf{0}\\
\bf{0}& \bf{1}
\end{bmatrix}$
\end{center}
Since $\bf{1}$ is the maximal element of the incline $L$, so $a_{12}^{\otimes 2} \oplus \bf{1} = \bf{1}$.   Hence the CP-rank$(A) \leq 2$.

Now for $n = 3$, let us consider \begin{center}
$A = \begin{bmatrix}
\bf{1} & a_{12}& a_{13}\\
a_{12}& \bf{1} & a_{23}\\
a_{13} & a_{23} & \bf{1}
\end{bmatrix}$
\end{center}
then the rank 1 CP-representation of $A$ is
\begin{center}
$A = \begin{bmatrix}
\bf{1}\\a_{12}\\ \bf{0}
\end{bmatrix}\begin{bmatrix}
\bf{1}& a_{12}&\bf{0}
\end{bmatrix} \bigoplus \begin{bmatrix}
\bf{0}\\\bf{1}\\a_{23}
\end{bmatrix}\begin{bmatrix}
\bf{0}&\bf{1}& a_{23}\\
\end{bmatrix}$
$\bigoplus
 \begin{bmatrix}
a_{13}\\ \bf{0}\\\bf{1}
\end{bmatrix}\begin{bmatrix}
a_{13}& \bf{0}& \bf{1}\\
\end{bmatrix}$
\end{center}
\begin{center}
= $\begin{bmatrix}
\bf{1}& a_{12}& \bf{0}\\
a_{12}& a_{12}^{\otimes 2}& \bf{0}\\
\bf{0}&\bf{0}&\bf{0}\\
\end{bmatrix} \bigoplus \begin{bmatrix}
\bf{0}& \bf{0}& \bf{0}\\
\bf{0}& \bf{1} & a_{23}\\
\bf{0}&a_{23}& a_{23}^{\otimes 2}\\
\end{bmatrix} \bigoplus \begin{bmatrix}
a_{13}^{\otimes 2}& \bf{0}& a_{13}\\
\bf{0}& \bf{0}& \bf{0}\\
a_{13}&\bf{0}& \bf{1}\\
\end{bmatrix}$
\end{center}
Since $\bf{1}$ is the maximal element of the   incline $L$, the diagonal entries of $A$ will not be affected when we will add all rank one completely positive  matrices in the rank 1 CP-representation of $A$. Hence the CP-rank$(A) \leq 3$, furthermore, equality holds for nonzero diagonal matrices of order $2 \times 2$ and $3 \times 3$. Thus the CP-rank$(A) \leq n$ for $n = 2,3$.\\
\end{proof}

\begin{thm} \label{inclines_CP_DJL_result}  Let $L$ be a totally ordered normal  incline  and $A$ be an $n \times n$ completely positive matrix over $L$. Then the  CP-rank$(A) \leq max\{n,[n^{2}/4]\}$. Further, $A$ has an $[n^{2}/4]$ - support 3 rank 1 CP-representation for $n \geq 4$.
\end{thm}
\begin{proof} Let us suppose that  $A$ is a completely positive matrix over  a totally ordered normal  incline $L$. Because of theorem \ref{thmequivlantincline}, we need only to consider symmetric matrices over $L$ all of whose diagonal entries are $\bf{1}$.  If $n \leq 3$ then $max\{n,[n^{2}/4]\} = n$ and by Lemma \ref{inclines_cp_2_3}, the CP-rank$(A) \leq n$. Now we will prove the theorem for $n \geq 4$. First, we will show that the result is true for $n = 4$ and $5$, then we will use induction going from index $n$ to index $n+2$.

For $n = 4$,  let us consider \begin{center}
$A = \begin{bmatrix}
\bf{1} & a_{12}& a_{13}&a_{14}\\
a_{12}& \bf{1} & a_{23}&a_{24}\\
a_{13} & a_{23} & \bf{1} &a_{34}\\
a_{14} & a_{24} & a_{34}& \bf{1} \\
\end{bmatrix}$
\end{center}
 Without loss of generality, we assume that $a_{13}$ is the smallest non-diagonal entry of the entire matrix and $a_{12}$ is the largest non-diagonal entry in the first row. Then the rank 1 CP-representation of $A$ is
\begin{center}
$ A = \begin{bmatrix}
\bf{1}\\a_{12}\\a_{13}\\\textbf{0}
\end{bmatrix}
\begin{bmatrix}
\bf{1}& a_{12}&a_{13}&\textbf{0}
\end{bmatrix}  \bigoplus \begin{bmatrix}
\textbf{0}\\\bf{1}\\a_{23}\\\textbf{0}
\end{bmatrix}   \begin{bmatrix}
\textbf{0}& \bf{1} &a_{23}&\textbf{0}\\
\end{bmatrix}$\end{center}  \begin{flushright}$
  \bigoplus \begin{bmatrix}
\textbf{0}\\\textbf{0}\\\bf{1}\\a_{34}
\end{bmatrix}
\begin{bmatrix}
\textbf{0}&\textbf{0}& \bf{1}&a_{34}\\
\end{bmatrix}
  \bigoplus \begin{bmatrix}
a_{14}\\a_{24}\\\textbf{0}\\\bf{1}
\end{bmatrix}
\begin{bmatrix}
a_{14}& a_{24}&\textbf{0}&\bf{1}\\
\end{bmatrix}$
\end{flushright}

\[
= \begin{bmatrix}
\bf{1}& a_{12}& a_{13}& \textbf{0}\\
a_{12}& a_{12}^{\otimes 2}& a_{12}\otimes a_{13}& \textbf{0}\\
 a_{13}& a_{12}\otimes a_{13}& a_{13}^{\otimes 2}& \textbf{0}\\
\textbf{0}& \textbf{0}& \textbf{0}& \textbf{0}\\
\end{bmatrix} \bigoplus \begin{bmatrix}
\textbf{0}& \textbf{0}& \textbf{0}& \textbf{0}\\
\textbf{0}& \bf{1}& a_{23}&\textbf{0}\\
\textbf{0}& a_{23}& a_{23}^{\otimes 2}& \textbf{0}\\
\textbf{0}& \textbf{0}& \textbf{0}& \textbf{0}\\
\end{bmatrix}
\]
\begin{flushright}
$\bigoplus \begin{bmatrix}
\textbf{0}& \textbf{0}& \textbf{0}& \textbf{0}\\
\textbf{0}& \textbf{0}& \textbf{0}& \textbf{0}\\
\textbf{0}& \textbf{0}& \bf{1}& a_{34}\\
\textbf{0}& \textbf{0}& a_{34}& a_{34}^{\otimes 2}\\
\end{bmatrix} \bigoplus \begin{bmatrix}
a_{14}& a_{14}\otimes a_{24} & \textbf{0}& a_{14}\\
a_{14}\otimes a_{24}& a_{24}^{\otimes 2}& \textbf{0}& a_{24}\\
\textbf{0}& \textbf{0}& \textbf{0}& \textbf{0}\\
a_{14}& a_{24}& \textbf{0}& \bf{1}\\
\end{bmatrix} $.
\end{flushright}
  Since $a_{12}\otimes a_{13} \leq a_{13} \leq a_{23}$ and $a_{14}\otimes a_{24}  \leq a_{14} \leq a_{12}$, so the addition of $a_{12}\otimes a_{13}$ and $a_{14}\otimes a_{24}$ will not affect the $(2,3)^{th}$ and $(1,2)^{th}$ entry of $A$  respectively. Hence CP-rank$(A) \leq 4 = [4^{2}/4]$  with support three.

Now for $n = 5$, let us consider \begin{center}
$A = \begin{bmatrix}
 \bf{1}& a_{12}& a_{13}&a_{14}&a_{15}\\
a_{12}& \bf{1}& a_{23}&a_{24}&a_{25}\\
a_{13} & a_{23} &  \bf{1}&a_{34}&a_{35}\\
a_{14} & a_{24} & a_{34}& \bf{1}&a_{45}\\
a_{15} & a_{25} & a_{35}&a_{45}& \bf{1}\\
\end{bmatrix}$
\end{center}
Without loss of generality, let us suppose that  $a_{13}$ is the smallest non-diagonal entry of the entire matrix and $a_{12}$ be the largest  non-diagonal entry in the first row. We need the following rank one completely positive   matrices for the rank 1 CP-representation of $A$.
\begin{itemize}
\item $B_{1} = \begin{bmatrix}
 \bf{1}\\a_{12}\\a_{13}\\\textbf{0}\\\textbf{0}\\
\end{bmatrix}\begin{bmatrix}
 \bf{1}&a_{12}&a_{13}&\textbf{0}&\textbf{0}
\end{bmatrix} = \begin{bmatrix}
 \bf{1}&a_{12}& a_{13}&\textbf{0}&\textbf{0}\\
a_{12}&a_{12}^{\otimes 2}&a_{12}\otimes a_{13}&\textbf{0}&\textbf{0}\\
a_{13}&a_{12}\otimes a_{13}&a_{13}^{\otimes 2}&\textbf{0}&\textbf{0}\\
 \textbf{0}&\textbf{0}&\textbf{0}&\textbf{0}&\textbf{0} \\
  \textbf{0}&\textbf{0}&\textbf{0}&\textbf{0}&\textbf{0} \\
  \end{bmatrix} $
\item $B_{2} = \begin{bmatrix}
\textbf{0}\\\bf{1}\\a_{23}\\\textbf{0}\\\textbf{0}\\
\end{bmatrix}\begin{bmatrix}
\textbf{0}&\bf{1}&a_{23}&\textbf{0}&\textbf{0}
\end{bmatrix}  = \begin{bmatrix}
\textbf{0}&\textbf{0}& \textbf{0}&\textbf{0}&\textbf{0}\\
\textbf{0}&\bf{1}&a_{23}&\textbf{0}&\textbf{0}\\
 \textbf{0}&a_{23}&a_{23}^{\otimes 2}&\textbf{0}&\textbf{0}\\
 \textbf{0}&\textbf{0}&\textbf{0}&\textbf{0}&\textbf{0} \\
  \textbf{0}&\textbf{0}&\textbf{0}&\textbf{0}&\textbf{0} \\
  \end{bmatrix} $
\item $B_{3} = \begin{bmatrix}
a_{14}\\a_{24}\\\textbf{0}\\\bf{1}\\\textbf{0}\\
\end{bmatrix}\begin{bmatrix}
a_{14}&a_{24}&\textbf{0}&\bf{1}&\textbf{0}
\end{bmatrix}  = \begin{bmatrix}
a_{14}^{\otimes 2}& a_{14}\otimes a_{24} &\textbf{0}&a_{14}&\textbf{0}\\
 a_{14}\otimes a_{24} &a_{24}^{\otimes 2}&\textbf{0}&a_{24}&\textbf{0}\\
\textbf{0}&\textbf{0}&\textbf{0}&\textbf{0}&\textbf{0}\\
a_{14}& a_{24}&\textbf{0}&\bf{1}&\textbf{0} \\
  \textbf{0}&\textbf{0}&\textbf{0}&\textbf{0}&\textbf{0} \\
  \end{bmatrix} $
\item $B_{4} = \begin{bmatrix}
a_{15}\\a_{25}\\\textbf{0}\\\textbf{0}\\\bf{1}\\
\end{bmatrix}\begin{bmatrix}
a_{15}&a_{25}&\textbf{0}&\textbf{0}&\bf{1}
\end{bmatrix}  = \begin{bmatrix}
a_{15}^{\otimes 2}& a_{15}\otimes a_{25} &\textbf{0}&\textbf{0}&a_{15}\\
 a_{15}\otimes a_{25} &a_{25}^{\otimes 2}&\textbf{0}&\textbf{0}&a_{25}\\
\textbf{0}&\textbf{0}&\textbf{0}&\textbf{0}&\textbf{0}\\
 \textbf{0}&\textbf{0}&\textbf{0}&\textbf{0}&\textbf{0} \\
  a_{15}&a_{25}&\textbf{0}&\textbf{0}&\bf{1} \\
  \end{bmatrix} $
\end{itemize}

Clearly $A = B_{1} \bigoplus B_{2} \bigoplus B_{3} \bigoplus B_{4} \bigoplus C$, where
\begin{center}
$C =  \begin{bmatrix}
\textbf{0}& \textbf{0}& \textbf{0}&\textbf{0}&\textbf{0}\\
\textbf{0}& \textbf{0}& \textbf{0}&\textbf{0}&\textbf{0}\\
\textbf{0} & \textbf{0} &\bf{1}&a_{34}&a_{35}\\
\textbf{0} & \textbf{0} & a_{34}&*&a_{45}\\
\textbf{0}& \textbf{0}& a_{35}&a_{45}&**\\
\end{bmatrix}$.
\end{center}

Here $*$ and $**$ can be any element  of the incline $L$. Now choose the smallest of $a_{34}, a_{35},$ and $a_{45}$.  We have the following cases:

Case 1: If $a_{34}$ or $a_{35}$ is the smallest of all the entries $\{a_{34}, a_{35}, a_{45}\}$ then $C$ can be written as the sum of two rank one completely positive matrices $C_{1}$ and $C_{2}$, where
\begin{center}
$C_{1} = \begin{bmatrix}
\textbf{0}\\\textbf{0}\\\bf{1}\\a_{34}\\a_{35}\\
\end{bmatrix}\begin{bmatrix}
\textbf{0}&\textbf{0}&\bf{1}&a_{34}&a_{35}
\end{bmatrix} = \begin{bmatrix}
\textbf{0}&\textbf{0}&\textbf{0}&\textbf{0}&\textbf{0}\\
\textbf{0}&\textbf{0}&\textbf{0}&\textbf{0}&\textbf{0}\\
\textbf{0}&\textbf{0}&\bf{1}&a_{34}&a_{35}\\
\textbf{0}&\textbf{0}&a_{34}&a_{34}^{\otimes 2}&a_{34}\otimes a_{35}\\
\textbf{0}&\textbf{0}&a_{35}&a_{34}\otimes a_{35}&a_{35}^{\otimes 2}\\
\end{bmatrix}$
\end{center}
and
\begin{center}$C_{2} = \begin{bmatrix}
\textbf{0}\\\textbf{0}\\\textbf{0}\\a_{45}\\\bf{1}\\
\end{bmatrix}\begin{bmatrix}
\textbf{0}&\textbf{0}&\textbf{0}&a_{45}&\bf{1}
\end{bmatrix} = \begin{bmatrix}
\textbf{0}&\textbf{0}&\textbf{0}&\textbf{0}&\textbf{0}\\
\textbf{0}&\textbf{0}&\textbf{0}&\textbf{0}&\textbf{0}\\
\textbf{0}&\textbf{0}&\textbf{0}&\textbf{0}&\textbf{0}\\
\textbf{0}&\textbf{0}&\textbf{0}&a_{45}^{\otimes 2}&a_{45} \\
\textbf{0}&\textbf{0}&\textbf{0}&a_{45} &\bf{1}\\
\end{bmatrix}$
\end{center}
Note that either we have $a_{34}\otimes a_{35} \leq a_{34} \leq a_{45}$ or $a_{34}\otimes a_{35} \leq a_{35} \leq a_{45}$, so the addition of $a_{34}\otimes a_{35}$ will not affect the $(4,5)^{th}$ entry of $A$ when we will add $C_{1}$ and $C_{2}$ to the rank 1 CP-representation of $A$.

Case 2: If $a_{45}$ is the smallest of all the entries $\{a_{34}, a_{35}, a_{45}\}$ then $C$ can be written as the sum of two rank one completely positive  matrices $C_{1}$ and $C_{2}$, where
\begin{center}
$C_{1} = \begin{bmatrix}
\textbf{0}\\\textbf{0}\\\bf{1}\\a_{34}\\a_{45}\\
\end{bmatrix}\begin{bmatrix}
\textbf{0}&\textbf{0}&\bf{1}&a_{34}&a_{45}
\end{bmatrix}  = \begin{bmatrix}
\textbf{0}&\textbf{0}&\textbf{0}&\textbf{0}&\textbf{0}\\
\textbf{0}&\textbf{0}&\textbf{0}&\textbf{0}&\textbf{0}\\
\textbf{0}&\textbf{0}&\bf{1}&a_{34}&a_{45}\\
\textbf{0}&\textbf{0}&a_{34}&a_{34}^{\otimes 2}&a_{34}\otimes a_{45}\\
\textbf{0}&\textbf{0}&a_{45}&a_{34}\otimes a_{45}&a_{45}^{\otimes 2}\\
\end{bmatrix}$
\end{center} and
\begin{center} $C_{2} = \begin{bmatrix}
\textbf{0}\\\textbf{0}\\a_{35}\\a_{45}\\\bf{1}\\
\end{bmatrix}\begin{bmatrix}
\textbf{0}&\textbf{0}&a_{35}&a_{45}&&\bf{1}
\end{bmatrix}  = \begin{bmatrix}
\textbf{0}&\textbf{0}&\textbf{0}&\textbf{0}&\textbf{0}\\
\textbf{0}&\textbf{0}&\textbf{0}&\textbf{0}&\textbf{0}\\
\textbf{0}&\textbf{0}&a_{35}^{\otimes 2}&a_{35}\otimes a_{45}&a_{35}\\
\textbf{0}&\textbf{0}&a_{35}\otimes a_{45}&a_{45}^{\otimes 2}&a_{45}\\
\textbf{0}&\textbf{0}&a_{35}&a_{45}&\bf{1}\\
\end{bmatrix}$
\end{center}
Note that   we have   $a_{34}\otimes a_{45} \leq a_{45} \leq a_{34}$, so it will not effect the $(3,4)^{th}$   entry of $A$ and   $a_{45} \leq a_{35}$, so  it will not affect the $(3,5)^{th}$ entry of $A$  when we will add $C_{1}$ and $C_{2}$ to the rank 1 CP-representation of $A$.

In both cases $C$ can be written as the sum of two rank one completely positive  matrices. Thus the rank 1 CP-representation of $A$ is
\begin{center}
$A = B_{1}  \bigoplus B_{2} \bigoplus B_{3} \bigoplus B_{4} \bigoplus C_{1} \bigoplus C_{2}$
\end{center}
Hence the CP-rank$(A) \leq 6 = [5^{2}/4]$ with support three.

Thus the theorem is true for $n = 4,5$.  Further, we note that for any integer $n$, \begin{center} $\left[(n + 2)^{2}/4\right] = \left[n^{2}/4\right] + n + 1$. \end{center}

Now we are ready to prove the induction step. Let $A_{n+2}$ be a  symmetric matrices over $L$ all of whose diagonal entries are $\bf{1}$.  Without loss of generality,  let  us suppose that  $a_{13}$ is the smallest non-diagonal entry of the entire matrix and $a_{12}$ is the largest non-diagonal entry  in the first row.  Let $A_{n}$ be a submatrix of $A_{n+2}$ obtained by deleting $1^{st}$ and $2^{nd}$ row and $1^{st}$ and $2^{nd}$ column of $A_{n+2}$. Since $A_{n}$ is a principal submatrix of $A_{n+2}$, $A_{n}$ is also a   symmetric    matrix over $L$ all of whose diagonal entries are $\textbf{1}$. By the induction hypothesis, $A_{n}$ is the sum of at most $\left[n^{2}/4\right]$ rank  one completely positive matrices with support three.

Now for each $i$, where $i =  4,5,....n+2$, we introduce a rank one completely positive matrix \begin{center}$B_{i} = b_{i}b_{i}^{T} = \begin{bmatrix}
a_{1i}\\a_{2i}\\\textbf{0}\\.\\.\\\textbf{0}\\a_{ii} = \bf{1} \\\textbf{0}\\.\\\textbf{0}\\
\end{bmatrix} \begin{bmatrix}
a_{1i}&a_{2i}&\textbf{0}&.&.&\textbf{0}& a_{ii} = \bf{1} &\textbf{0}&.&\textbf{0}
\end{bmatrix}$ \end{center}

Note that here we get the $(1,2)^{th}$ entry  of $(B_{i}B_{i}^{T})$ is  $a_{1i}\otimes a_{2i}$, which is less than or equal to $a_{1i} \leq a_{12}$. Thus it will not affect the $(1,2)^{th}$ entry of $A_{n+2}$ when we will add the $b_{i}b_{i}^{T}$ term to the rank 1 CP representation of $A_{n+2}$. However this rank one completely positive matrix  fixes $a_{1i}$, $a_{2i}$ and $a_{ii} = \bf{1}$ in $A_{n+2}$. Hence we have at most $(n+2)-3$ rank one completely positive  matrices.

Finally, we need two  rank one completely positive matrices:
\begin{itemize}
\item $B_{1} = b_{1}b_{1}^{T} = \begin{bmatrix}
\bf{1}\\a_{12}\\a_{13}\\\textbf{0}\\.\\.\\\textbf{0}
\end{bmatrix} \begin{bmatrix}
\bf{1}&a_{12}&a_{13}&\textbf{0}&.&.&\textbf{0}
\end{bmatrix}$. \end{itemize}
Since $a_{13}$ is the smallest non-diagonal entry, it will not affect the $(2,3)^{th}$ entry of $A_{n+2}$ when we will add the $b_{1}b_{1}^{T}$ term to the rank 1 CP representation of $A_{n+2}$. However this rank one completely positive  matrix fixes $a_{11} = \bf{1}$, $a_{12}$ and $a_{13}$ in $A_{n+2}$.\begin{itemize}
\item  $B_{2} = b_{2}b_{2}^{T} = \begin{bmatrix}
\textbf{0}\\\bf{1}\\a_{23}\\\textbf{0}\\.\\.\\\textbf{0}
\end{bmatrix} \begin{bmatrix}
\textbf{0}&\bf{1}&a_{23}&\textbf{0}&.&.&\textbf{0}\\
\end{bmatrix}$. \end{itemize}
This rank one completely positive  matrix will fix $a_{22} = \bf{1}$ and $a_{23}$ in $A_{n+2}$.

Thus $A_{n+2}$ is the sum of at most
\begin{center}
$\left[n^{2}/4\right] + (n+2)-3 + 2$ = $[n^{2}/4] + n + 1$
\end{center}rank 1 CP-matrices with support 3.

Further we will show that $\left[n^{2}/4\right]$ can not be replaced by any smaller number. The Boolean semiring is a totally ordered normal incline and it has been shown  \cite[Remark~ 3.1]{PMTA} that  the upper bound of the CP-rank is achieved for Boolean matrices. Hence $\left[n^{2}/4\right]$ can not be replaced by any smaller number.
\end{proof}

Since the Boolean semiring and  max-min semirings are totally ordered normal inclines, our proof generalizes the results proved in \cite{PMTA} for completely positive matrices over max-min semirings.

\section{LU $\&$ UL Factorization of CP Matrices Over Normal Inclines}\label{LU_UL_factorization_inclines}

In this section, we generalize various results for completely positive matrices over reals to completely positive matrices over normal inclines.   These results give conditions which guarantee that a completely positive matrix over  a normal incline  has a square factorization $BB^{T}$ especially one where $B$ is a triangular matrix.

  We first  define $UL$ and $LU$ completely positive matrices over inclines.  Our definition is a natural generalization of real $UL$ and $LU$ completely positive matrices.

\begin{defn} $(\text{UL and LU-completely positive matrix})$  Let $L$ be a commutative incline. A matrix $A$ over   $L$  is called a \textit{UL-completely positive matrix} if there exists an upper triangular matrix $U$ over $L$ such that $A = UU^{T}$. A matrix $A$ over $L$ is called a \textit{LU-completely positive matrix} if there exists a lower triangular matrix $C$ over $L$ such that $A = CC^{T}$.
\end{defn}

\begin{exam} \label{maxep2} We note that every $2 \times 2$ completely positive matrix over  a normal  incline $L$  is both LU-completely positive and UL-completely positive.  This can be proved directly.  
  Let  \begin{center}
 $A = \begin{bmatrix}
a_{11}& a_{12}\\
a_{12}& a_{22}\\
\end{bmatrix}$
\end{center}
be a completely positive matrix over a normal incline $L$. Then by theorem \ref{thmequivlantincline},  every $2$ by $2$ principal submatrix of $A$ has  the positive determinant   greater than or equal to the negative determinant, i.e.,  $a_{11}\otimes a_{22} \geq a_{12}^{\otimes 2}$. This implies that     $\sqrt{a_{11}}\otimes \sqrt{a_{22}} \geq a_{12}$, since in a normal incline the square root function is increasing and multiplicative. Now using the  LI-property, we get that there exists $c \in L$ such that $c \sqrt{a_{11}}\otimes \sqrt{a_{22}}  = a_{12}$. Now the matrix $A$ can be written as:
\begin{center}
$A = \begin{bmatrix}
\sqrt{a_{11}} & c \sqrt{a_{11}} \\\textbf{0} & \sqrt{a_{22}}
\end{bmatrix}\begin{bmatrix}
\sqrt{a_{11}} & c \sqrt{a_{11}} \\\textbf{0} & \sqrt{a_{22}}
\end{bmatrix}^{T}$
\end{center}

and similarly,

\begin{center}
$A = \begin{bmatrix}
\sqrt{a_{11}} & \textbf{0} \\ c \sqrt{a_{22}} & \sqrt{a_{22}}
\end{bmatrix}\begin{bmatrix}
\sqrt{a_{11}} & \textbf{0} \\ c \sqrt{a_{22}} & \sqrt{a_{22}}
\end{bmatrix}^{T}$
\end{center} \end{exam}

Further, we will see that   $3 \times 3$ completely positive matrices over   normal inclines  are either UL-completely positive or LU-completely positive, but not necessarily both. However, for $n \geq 4$,   $n \times n$ completely positive matrices over normal  inclines may be neither UL-completely positive nor LU-completely positive.

We begin by proving   incline analogs of Markham theorems (theorem \ref{Mar1} and theorem \ref{Mar2}) relating almost principal minors with the triangular factorizations.

\begin{thm} \label{maxthm1} If $A$ is an $n \times n$ completely positive matrix over  a normal  incline $L$, $n \geq 3$, and all   its left  almost  principal $2 \times 2$ submatrices   have $det^{+} \geq det^{-}$, then $A$ is UL-completely positive.\end{thm}
\begin{proof}  Let $A$ be an $n \times n$ completely positive matrix over  a normal incline $L$. Because of theorem \ref{thmequivlantincline}, we assume that $A$ is a symmetric matrix over $L$ whose  diagonal entries are  equal to $\textbf{1}$. We will prove this theorem by constructing an upper  triangular matrix $U$ such that $A = UU^{T}$. Let
\begin{center}
$u_{ij} = \left\{
\begin{array}{l l}
a_{ij}   \quad \quad \qquad \qquad \text{ if  } i < j\\

\textbf{1} \quad \quad \quad \quad  \quad \quad \text{ if } i = j\\

\textbf{0} \quad   \qquad \quad \quad \quad\text{otherwise}.\\
\end{array} \right.$
\end{center}
Then we have
\begin{center}
$(UU^{T})_{ii} = \bigoplus\limits_{k=1}^{n} (u_{ik} \otimes u_{ik})$.
\end{center}
 It is clear from the construction of $U$ that $u_{ik}$ is nonzero ($\neq \textbf{0}$) only if $k \geq i$. We  split the above  sum into two parts. In the first part the sum is over all $k > i$ and   the second part consists of the single term where $k = i$. Clearly the second part contains a single entry of the sum. Thus we have
\begin{align*}
 (UU^{T})_{ii} &= \bigoplus\limits_{k > i} \quad  (u_{ik} \otimes u_{ik})\quad  \oplus \quad (u_{ii} \otimes u_{ii})\\
&=  \mathop{l.u.b}\limits_{k > i}^{} \quad (a_{ik} \otimes a_{ik})  \quad \oplus \quad (\textbf{1} \otimes \textbf{1})\\
&= \textbf{1}, \qquad (\text{since $\textbf{1}$ is the greatest element of the incline}).
 \end{align*}

Now for $i \neq j$,  \begin{center}
$(UU^{T})_{ij} = \bigoplus\limits_{k=1}^{n} (u_{ik} \otimes u_{jk})$.
\end{center}
 Without loss of generality, we assume that $i < j$. It is clear from the construction of $U$ that $u_{ik}$ is nonzero only if $k \geq i$ and $u_{jk}$ is nonzero only if $k \geq   j$. This implies that  $(u_{ik} \otimes u_{jk})$ is nonzero only if $k \geq j$. Now we split the above sum into two parts. In the first part the sum is over all $k > j$ and  the second part consists of the single term where $k = j$. Clearly the second part contains a single entry of the sum. Thus we have,
\begin{align*}
 (UU^{T})_{ij} &=  \bigoplus\limits_{k > j} (u_{ik} \otimes u_{jk})\quad  \oplus \quad (u_{ij} \otimes u_{jj})\\
&=  \mathop{l.u.b}\limits_{k > j}^{}    (a_{ik} \otimes a_{jk}) \quad   \oplus  \quad  (a_{ij} \otimes \textbf{1})\\
&=  \mathop{l.u.b}\limits_{k > j}^{}  (a_{ik} \otimes a_{jk}) \quad   \oplus  \quad   a_{ij}
\end{align*}
Since all the left  almost  principal $2 \times 2$ submatrices of $A$  have $det^{+} \geq det^{-}$, we have for all $k$, where $k > j > i$, $det^{+}(A[i,k|j,k]) \geq det^{-}(A[i,k|j,k])$, i.e., $a_{ij} \otimes a_{kk} \geq a_{ik} \otimes a_{jk}$. This implies that $ a_{ij} \otimes \textbf{1}   = a_{ij} \geq a_{ik} \otimes a_{jk}$ and thus the
\begin{center}$ (UU^{T})_{ij}=\mathop{l.u.b}\limits_{k > j}^{}   (a_{ik} \otimes a_{jk}) \oplus   a_{ij}  = a_{ij}$.\end{center}
This proves that $A = UU^{T}$. Hence $A$ is an UL-completely positive matrix.
\end{proof}

Analogous results hold for LU-completely positive matrices. By a similar argument to the previous theorem, we have:

\begin{thm} \label{maxthm2} If $A$ is an $n \times n$ completely positive matrix over  a normal  incline $L$ and all its right  almost  principal $2 \times 2$ submatrices  have $det^{+} \geq det^{-}$, then $A$ is LU-completely positive.\end{thm}

The converses of theorem \ref{maxthm1} and theorem \ref{maxthm2} are not true. We give a counterexample of a UL-completely positive matrix over  a  normal incline $L$    which has a left almost principal $2 \times 2$ submatrix  that does not satisfy the inequality $det^{+} \geq det^{-}$. Since any max-min semiring is a normal incline, we use matrices over a max-min semiring $[0,1]$ in the following example.
\begin{exam}
Let
\begin{center}
 $A = \begin{bmatrix}
1 & 0.25 & 0.5 & 0\\
0.25 & 0.75 & 0.5 & 0\\
0.5 &0.5 & 0.75 & 0 \\
0 & 0 & 0 & 0 \\
\end{bmatrix}$
\end{center}
be a  $4 \times 4$ matrix over the $[0,1]$ max-min semiring. The matrix $A$ is UL-completely positive because  there exists an upper triangular matrix $U$ such that $A = UU^{T}$, where
\begin{center}
$U = \begin{bmatrix}
0 & 1 & 0 & 0.5 \\
0 & 0.25 & 0.75 &0 \\
0 & 0 & 0.5 & 0.75\\
0 & 0 & 0 & 0 \\
\end{bmatrix}$
\end{center}
However, the left almost principal $2 \times 2$ submatrix $A[2,3|1,3]$ of  $A$ has $det^{+} = 0.25 \leq 0.5 = det^{-}$.
\end{exam}
Theorems  \ref{maxthm1} and \ref{maxthm2} are   incline  generalizations of theorems \ref{Mar1} and \ref{Mar2}. We note that in the incline case, we need only $2 \times 2$  left or right  almost  principal  submatrices, i.e.,  if all $2 \times 2$ left or right  almost  principal  submatrices of an $n \times n$ completely positive matrix have  $det^{+} \geq det^{-}$ then the matrix is UL or LU completely positive respectively.

\begin{remark} The CP-rank of an $n \times n$ completely positive matrix over a  normal incline $L$  is less than or equal to $n$ if either all its left  almost  principal $2 \times 2$ submatrices  or all of its right almost  principal $2 \times 2$ submatrices have $det^{+} \geq det^{-}$. \end{remark}

 Now we will prove that all  $3 \times 3$ completely positive matrices over totally ordered normal inclines  are either LU-completely positive matrices or UL-completely positive matrices (or both). We need the following lemma:

\begin{lemma} \label{maxlemma1} Let $L$ be a totally ordered normal incline. If $A$ is a $3 \times 3$ completely positive matrix over   $L$, then at least two of the following inequalities holds:
\begin{center}
$a_{11} \otimes a_{23} \geq a_{12} \otimes a_{13}$\\
$a_{22} \otimes a_{13} \geq a_{12} \otimes a_{23}$\\
$a_{33} \otimes a_{12} \geq a_{13} \otimes a_{23}$
\end{center}
\end{lemma}
\begin{proof} Suppose that two of the inequalities do not hold, say we have instead,
\begin{center}
$a_{11} \otimes a_{23} < a_{12} \otimes a_{13}$\\$ a_{22} \otimes a_{13} < a_{12} \otimes a_{23}$.
\end{center}
Then \begin{equation} \label{equatuon_lemma_maxlemma1}
 a_{11} \otimes a_{23} \otimes a_{22} \otimes a_{13} < a_{12} \otimes a_{13} \otimes a_{12} \otimes a_{23}.
\end{equation}
Given that   $A$ is a $3 \times 3$ completely positive matrix over   a normal incline $L$. Therefore, by theorem \ref{thmequivlantincline}, every $2 \times 2$ principal submatrix of $A$ has $det^{+} \geq det^{-}$, i.e.,  $a_{11} \otimes a_{22} \geq  a_{12}^{\otimes 2}$. Therefore,   $a_{11} \otimes a_{23} \otimes a_{22} \otimes a_{13} \geq a_{12} \otimes a_{13} \otimes a_{12} \otimes a_{23}$, where $a_{23}, a_{13} \in L$. Thus we get a contradiction to equation \eqref{equatuon_lemma_maxlemma1}.  All other cases can be proved using a similar argument.
\end{proof}

Now we will relate inequalities of the above lemma with  the positive and the negative determinant of $2 \times 2$ submatrices of the given matrix $A$. In the above lemma, the first inequality $a_{11} \otimes a_{23} \geq a_{12} \otimes a_{13}$ implies that  $det^{+}(A[1,2|1,3]) \geq det^{-}(A[1,2|1,3])$.
In  other words, this inequality implies that the right almost  principal $2 \times 2$ submatrix $A[1,2|1,3]$ of $A$ has the  positive determinant greater than or equal to the negative determinant.

The third inequality $a_{33} \otimes a_{12} \geq a_{13}\otimes a_{23}$ implies that  $det^{+}(A[1,3|2,3]) \geq det^{-}(A[1,3|2,3])$, or stated in words the left almost  principal $2 \times 2$ submatrix $A[1,3|2,3]$ of $A$ has the  positive determinant greater than or equal to the negative determinant. However the second  inequality $a_{22} \otimes a_{13} \geq a_{12} \otimes a_{23}$ has no relation with any left or right almost  principal $2 \times 2$ submatrix of $A$.

\begin{thm}\label{3_3_lu_or_Ul} Let $L$ be a totally ordered normal incline. If $A$ is a $3 \times 3$ completely positive matrix over   $L$, then $A$ is either LU-completely positive or UL-completely positive or both.\end{thm}
\begin{proof} If any diagonal entry of $A$ is $\textbf{0}$, then all the entries in the corresponding row and column will be $\textbf{0}$. In this case the result follows from example \ref{maxep2}. Now suppose that $a_{ii} > \textbf{0}$ for all $i = 1,2,3$. The only left almost principal $2 \times 2$ submatrices of $A$ are
\begin{center}
$A_{1} = A[1,3|2,3]  \qquad \qquad A_{2} = A[2,3|1,3]$
\end{center}
and
\begin{center}
$det^{+}(A_{1}) = det^{+}(A_{2}) = a_{33} \otimes a_{12}$ \\
$det^{-}(A_{1}) = det^{-}(A_{2}) =  a_{13} \otimes a_{23}$
\end{center}
and the only right  almost principal $2 \times 2$ submatrices of $A$ are
\begin{center}
$A_{3} = A[1,3|1,2]  \qquad \qquad A_{4} = A[1,2|1,3]$
\end{center}
and
\begin{center}
$det^{+}(A_{3}) =  det^{+}(A_{4}) = a_{11} \otimes a_{23}$ \\
$det^{-}(A_{3}) = det^{-}(A_{4}) =  a_{13} \otimes a_{12}$
\end{center}
By lemma \ref{maxlemma1}, either the left almost principal $2 \times 2$ submatrices have $det^{+} \geq det^{-}$ or the right almost principal $2 \times 2$ submatrices have $det^{+} \geq det^{-}$ or both. Thus by theorem \ref{maxthm1} and \ref{maxthm2},   $A$ is either  UL-completely positive or LU-completely positive or both.
\end{proof}
\begin{exam} Let $L$ be a totally ordered normal incline.  A $3 \times 3$ matrix $A$ over   $L$   need not be both LU-completely positive and UL-completely positive. For example, consider the matrix $A$ over the $[0,1]$ max-min semiring, where \begin{center}
$A = \begin{bmatrix}
0.75& 0 & 0.25 \\ 0 & 0.5 & 0.5 \\0.25 & 0.5 &1
\end{bmatrix}$
\end{center} Every right almost principal $2 \times 2$ submatrix of $A$ has its positive determinant greater than or equal to their negative determinant. Thus by theorem \ref{maxthm2}, $A$ is LU-completely positive. However, we can check that there does not exist  any upper triangular matrix $U$ over the max-min semiring such that $A = UU^{T}$.\end{exam}

\begin{defn} $(TN_2 \text{ Matrix})$  A matrix over any commutative semiring is called a \textit{$TN_2$ matrix} if $det^+\ge det^-$ for all its two by two submatrices, i.e.,   for all $i, j, k, l$ with $i<k$ and $j<l$, we have  $a_{ij} \otimes a_{kl}\geq a_{il} \otimes a_{kj}$.
\end{defn}
 The class of $TN_2$ real matrices has interesting properties in its own right; the results in \cite{JN} are a good example of this. For $TN_2$  matrices over normal inclines, we have an analog of corollary \ref{marcor} for inclines.

\begin{cor} Every square symmetric $TN_2$ matrix over a  normal incline    is both LU- and UL- completely positive.
\end{cor}


\begin{thebibliography}{99}

\bibitem{AK02} S.~S.~Ahn and H.~S.~Kim, {\em On r-ideals in incline algebras,} Commun. Korean Math. Soc.,   {\bf 17}, No. 2, (2002), 229-235.

 \bibitem{BSM} A. Berman and N. Shaked-Monderer, {\em Remarks on completely positive matrices,}   Linear and Multilinear Algebra {\bf 44}, (1998), 149-163.

\bibitem{BSM03} A. Berman and N. Shaked-Monderer,  {\em Completely Positive Matrices,}  World Scientific, River Edge, NJ, (2003).

\bibitem{Bu03} P. Butkovi$\check{c}$,  {\em Max-algebra: the linear algebra of combinatorics$?$,}  Linear Algebra Appl. {\bf 367}, (2003), 313-335.


\bibitem{BSU14} I. M. Bomze, W. Schachinger and R. Ullrich,  {\em From seven to eleven: Completely positive matrices with high cp-rank,} Linear Algebra and Its Applications, {\bf 459}, (2014) 208-221.

\bibitem{BSU15} I. M. Bomze, W. Schachinger and R. Ullrich,  {\em New lower bounds and asymptotics for the CP rank,} SIAM J. Matrix Anal. {\bf 36}, (2015) 20-37.

\bibitem{Cao83} Z.~Q.~Cao,  {\em An algebraic system generalizing the fuzzy subsets of a set,}  Advances in fuzzy sets, possibility theory, and applications,  Plenum, New York, (1983), 71-80.






\bibitem{CKR84} Z.~Q.~Cao, K.~H.~Kim and F.~W.~Roush, {\em Incline Algebra and Applications,}  Ellis Horwood, Chichester, England, Wiley, New York, (1984).


\bibitem{DGZXT11} J-S.~Duan, A-P.~Guo, F-X.~Zhao, Li Xu and W-G.~Tang, {\em   Standard Bases of a Vector Space Over a Linearly Ordered Incline,}  Communications in Algebra,  {\bf 39:4}, (2011),  1404-1412, DOI: 10.1080/00927871003738915.



\bibitem{DrJ} J. H. Drew and C. R. Johnson, {\em The no long odd cycle theorem for completely positive matrices,} In: Random discrete structures, D. Aldous, R. Pemantle, Editors. IMA Vol. Math. Appl.,
vol. 76,  Springer, New York, (1996), 103-115.

\bibitem{DJL} J. H. Drew, C. R. Johnson, and R. Loewy, {\em  Completely positive matrices associated with M-matrices,} Linear and Multilinear Algebra {\bf 37}, (1994), 303-310.


\bibitem{GOL90} J. S. Golan, {\em Semirings for the ring theorist,}  Rev. Roumaine Math. Pures Appl.  {\bf 35},  (1990), 6, 531-540.



\bibitem{HaL} J. Hannah and T. L. Laffey, {\em Nonnegative factorization of completely positive matrices,}   Linear Algebra and Its Applications {\bf  55}, (1983), 1-9.

\bibitem{HR13} Song-Chol Han and  Hak-Rim Ri, {\em The only regular inclines are distributive lattices,}  Romanian Journal of Mathematics and Computer Science  {\bf 3},  Issue 2, (2013), 160-163.



\bibitem{JN} C. R. Johnson, and S. Nasserasr, {\em  $TP_2=$Bruhat,} Discrete Math. {\bf 310}, (2010), 1627-1628.

\bibitem{MKK} M. Kaykobad, {\em On Nonnegative Factorization of Matrices,}   Linear Algebra and Its Applications {\bf 96}, (1987), 23-33.


 \bibitem{KR04} K.~H.~Kim and F.~W.~Roush, {\em Inclines and incline matrices: a survey,}   Linear Algebra and Its Applications, {\bf 379}, (2004), 457-473.


\bibitem{LMa} C. M. Lau and T. L. Markham, {\em Square triangular factorizations of completely
positive matrices,}  J. Industrial Math. Soc.  {\bf  28}, (1978),15-24.


 \bibitem{LoT} R. Loewy and B-S. Tam, {\em CP rank of completely positive matrices of order five,}  Linear Algebra and Its Applications {\bf 363}, (2003), 161-176.

\bibitem{TLM} T. L. Markham, {\em Factorization of completely positive matrices,}   Proc. Cambridge
Philos. Soc. {\bf 69} (1971), 53-58.

\bibitem{PMTA} P. Mohindru, {\em The Drew-Johnson-Loewy conjecture for matrices over max-min semirings,}  Linear and Multilinear Algebra {\bf 63} (2015), 914–-926.

\bibitem{PMRP} P. Mohindru and R. Pereira, {\em Orderings on semirings and completely positive matrices,}  Linear and Multilinear Algebra,  {\bf 64}(2016), 818--833.

\bibitem{MPLUS90} M. Plus, {\em Linear systems in (max, +) algebra,}  In Proceedings of the 29th Conference on Decision and Control,  Honolulu, Dec. (1990).

\bibitem{PH04} Phillip L. Poplin and Robert E. Hartwig, {\em Determinantal identities over commutative semirings,} Linear Algebra and Its Applications  {\bf 387},  (2004), 99-132.

 \bibitem{NSM} N. Shaked-Monderer, {\em Minimal CP Rank,}   The Electronic Journal of Linear Algebra  {\bf 8}, (2001), 140-157.


\bibitem{SBJS} N. Shaked-Monderer, I. M. Bomze, F. Jarre and W. Schachinger, {\em On the CP-rank and minimal CP factorizations of a completely positive matrix,}  SIAM Journal of Matrix Analysis and Applications  {\bf 34}, No. 2, (2013), 355-368.

\bibitem{VAN34} H. S. Vandiver, {\em Note on a simple type of algebra in which the cancellation law of addition does not hold,} Bull. Amer. Math. Soc.   {\bf 40},  (1934), 914-920.

\bibitem{Zh08}  X. Zhan, {\em Open problems in matrix theory,}  Proceedings of the $4^{th}$ International Congress of Chinese Mathematicians,    {\bf 1},  (2008), 367-382.


\end{thebibliography}
\end{document}